\documentclass[12pt]{article}
\usepackage{graphicx}
\usepackage{xcolor}
\usepackage{amssymb,amsfonts,amsthm,amsmath,calligra,tikz}
\usepackage[utf8]{inputenc}
\usepackage{accents}
\usepackage{slashed}
\usepackage{yfonts}
\usepackage{mathrsfs,pifont}
\theoremstyle{definition}
\usepackage{tikz}
\usepackage[all]{xy}
\usepackage{mathtools}

\usepackage[style=numeric]{biblatex}
\usepackage{hyperref}

\newtheorem{Definition}{Definition}
\newtheorem{Proposition}{Proposition}
\newtheorem{Lemma}{Lemma}
\newtheorem{Corollary}{Corollary}

\newtheorem{Theorem}{Theorem}
\newtheorem{Remark}{Remark}

\let\bs\boldsymbol

\def\aa{{\bs a}}

\def\dd{{\bs d}}

\def\zz{{\bs z}}

\newcommand{\bA}{\mathsf{A}}

\newcommand{\bT}{\mathsf{T}}
\newcommand{\bK}{\mathsf{K}}
\newcommand{\qm}{\textnormal{\textsf{QM}}}

\newcommand{\Lie}{\mathrm{Lie}}
\newcommand{\Wall}{\mathsf{Wall}}
\newcommand{\Res}{\mathsf{Res}}
\newcommand{\Pic}{\mathrm{Pic}}
\newcommand{\Ind}{\mathrm{Ind}}
\newcommand{\stab}{\mathrm{Stab}}
\newcommand{\sympow}{{S}^{\scriptscriptstyle\bullet}}
\newcommand{\extpow}{{\bigwedge}^{\scriptscriptstyle\bullet}}
\def \vrs {\widehat{{\O}}_{{\rm{vir}}}}

\def\O {{\mathcal{O}}}

\title{Euler characteristic of stable envelopes}
\author{Hunter Dinkins and Andrey Smirnov}
\date{}

\begin{document}

\maketitle
\begin{abstract}
In this paper we prove a formula relating the equivariant Euler characteristic of $K$-theoretic stable envelopes to an object known as the index vertex for the cotangent bundle of the full flag variety. Our formula demonstrates that the index vertex is the power series expansion of a rational function. This result is a consequence of the 3d mirror self-symmetry of the variety considered here. In general, one expects an analogous result to hold for any two varieties related by 3d mirror symmetry. 
\end{abstract}

\tableofcontents

\section{Introduction}
Let $X$ be a smooth symplectic quasiprojective variety. We make the following assumptions:
\begin{itemize}
    \item A torus $\bT$ acts on $X$ such that it scales the symplectic form with character $\hbar$. We denote by $\bA:=\ker(\hbar)\subset\bT$ the subtorus preserving the symplectic form.
    \item The $\bT$-action has finitely many fixed points.
    \item The tangent bundle of $X$ has a polarization $T^{1/2}X$. In other words, there exists some class $T^{1/2}X \in K_{\bT}(X)$ so that the tangent bundle of $X$ decomposes as
    $$
    TX= T^{1/2}X + \hbar^{-1} \left(T^{1/2} X\right)^{\vee} \in K_{\bT}(X)
    $$
    \item A chamber $\mathfrak{C}\subset \Lie_{\mathbb{R}}(\bA)$ is fixed, which is a choice of connected component of the complement of
    $$
 \bigcup_{p\in X^{\bA}} \bigcup_{w\in \text{char}_{\bA}(T_pX)}\{\sigma\in \Lie_{\mathbb{R}}(\bA) \, \mid \, \langle \sigma, w \rangle =0\}
    $$
    where $\langle \cdot, \cdot \rangle$ is the pairing of characters and cocharacters. This chamber decomposes the tangent space at the fixed points into attracting and repelling directions.
\end{itemize}
Our main interest is the case when $X$ is a Nakajima quiver variety of linear or affine type $A$. If such a variety satisfies some natural conditions which hold for quiver varieties, it is known that the cohomological, $K$-theoretic, and elliptic stable envelopes exist, see \cite{AOElliptic} and \cite{indstab1}.

For a torus fixed point $p \in X^{\bT}$, the $K$-theoretic stable envelope provides a class
$$
\stab^{s,X,K}_{\mathfrak{C},T^{1/2}X}(p) \in K_{\bT}(X)
$$
which depends on the chosen polarization and chamber, as well as a generic choice of $s \in \Pic(X) \otimes_{\mathbb{Z}} \mathbb{R}$ called the slope. For an appropriately normalized version of the stable envelope, we consider the $K$-theoretic equivariant Euler characteristic
$$
\chi\left(\stab^{s,X,K}_{\mathfrak{C},T^{1/2}X}(p)\right) \in K_{\bT}(pt)_{loc}= \mathbb{C}(\aa,\hbar)
$$
where $\aa$ denotes the equivariant parameters of the torus $\bA$. A natural question is to ask for an explicit description of this rational function and to study its expansion as a series in $\aa$.

In this paper, we study this problem for the special case when $X$ is the cotangent bundle of the full flag variety. This variety is known to be self-dual with respect to 3d mirror symmetry, see \cite{MirSym2} and \cite{msflag}. We denote the 3d mirror dual copy of this variety by $X^{!}$ and denote the torus acting on this variety by $\bT^{!}$. While $X$ and $X^{!}$ are isomorphic as varieties, the 3d mirror symmetry relationship requires that various tori associated to $X$ and $X^{!}$ are related in non-trivial ways, see (\ref{kap}) below. Hence it will be necessary to distinguish them.

Using results from \cite{msflag}, we identify in Theorem \ref{mainthm} the power series expansion of $\chi\left(\stab^{s,X,K}_{\mathfrak{C},T^{1/2}X}(p)\right)$ around a certain point with an object we call the index vertex of $X^{!}$. We will now briefly explain the geometric meaning of the index vertex.

The index vertex arises in the enumerative geometry of quasimaps from $\mathbb{P}^1$ to $X^{!}$, see section 4 of \cite{pcmilect} and section 8.2 of \cite{NO}. For a $\bT^{!}$-fixed point $p^{!} \in X^{!}$, we denote by $\qm_{p^{!}}$ the moduli space of stable quasimaps from $\mathbb{P}^1$ to $X^{!}$ that evaluate to $p^{!}$ at $\infty$. The torus $\bT^{!}_q:=\bT^{!}\times\mathbb{C}^{\times}_q$, where $\mathbb{C}^{\times}_q$ acts on $\mathbb{P}^1$, acts on $\qm_{p^{!}}$ with a discrete fixed point set. A generic choice of the slope $s$ provides a decomposition of the virtual tangent space of $\qm_{p^{!}}$ at a $\bT^{!}_q$-fixed point into attracting and repelling directions. The index vertex $\Ind^{!s}_{p^{!}}$ is defined as the generating function that counts these repelling directions, with the sum taken over all possible degrees of quasimaps, see Definition \ref{ind} below. Then, roughly speaking, our first main theorem states:
$$
\chi\left(\stab^{s,X,K}_{\mathfrak{C},T^{1/2}X}(p)\right)= \Ind^{!s}_{p^{!}}
$$
where $p$ and $p^{!}$ are related under the bijection provided by 3d mirror symmetry as in section \ref{mirsym}.

As a consequence of this result, we deduce that the index vertex is the power series expansion of a rational function.

While $K$-theoretic stable envelopes require the slope $s$ to be generic, the index vertex can be defined for non-generic slopes $s$ as a limit of the so-called vertex function of $X^{!}$. In this case, the index vertex is still the power series expansion of a rational function, but the identification of the rational function is more complex. Using results from \cite{KS2}, we prove in Theorem \ref{mainthm2} that it gives the power series expansion of a certain rational function obtained from the stable envelopes of $X$ and stable envelopes of a subvariety $Y_s \subset X^{!}$.

The outline of this paper is as follows. In section \ref{basicnotions}, we describe some basic notions regarding the description of the cotangent bundle of the full flag variety as a Nakajima quiver variety. In section \ref{mirsym}, we describe a few pieces of data involved in the 3d mirror self-symmetry of this variety that will be important for us. In the next two sections we discuss the main objects of study in this paper: the index vertex in section \ref{indsection} and stable envelopes in \ref{stabsection}. We state and prove our main theorems in section \ref{theoremsection}. In the final section, we calculate all relevant quantities directly and verify our result for generic slopes explicitly in the simplest possible example.

We conclude this introduction with a few general remarks. Although the cotangent bundle of the full flag variety is a particularly nice variety, we expect the results of this paper to hold much more generally. More specifically, we expect that for any two varieties related by 3d mirror symmetry satisfying the conditions listed at the beginning of the introduction, the index vertex and the $K$-theoretic stable envelopes should be related in the same way as described here. Indeed, the proof of Theorem \ref{mainthm} shows that the relationship between the $K$-theoretic stable envelopes and the index vertex is really a consequence of a more general conjecture regarding elliptic stable envelopes and vertex functions of 3d mirror dual varieties, see \cite{dinkms1}, \cite{msflag}, \cite{mstoric}, \cite{liu} and the introduction of \cite{AOElliptic}. While a general construction of 3d mirror dual pairs is not presently known, the construction of Coloumb branches in \cite{coul1} and \cite{coul2} provides a large class of varieties 3d mirror dual to quiver varieties, and more generally Higgs branches of 3d $N=4$ gauge theories. In the case of type $A$, which closely resembles this paper, 3d mirror dual pairs can be realized as Cherkis bow varieties, see \cite{NakBow}, \cite{RSbows}, \cite{Cherk1}, \cite{Cherk2}, and \cite{Cherk3}.

\subsection{Acknowledgements}
This work was partially supported by NSF grant DMS-2054527.

\section{Basic properties of \texorpdfstring{$X$}{X}}\label{basicnotions}
\subsection{Description as a quiver variety}
We construct the cotangent bundle of the full flag variety as a Nakajima quiver variety. The quiver data is given in Figure \ref{fig1}. 
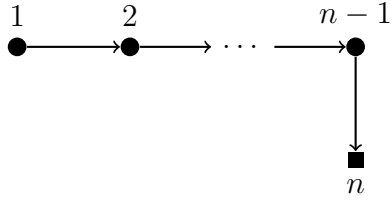
\begin{figure}[ht]
\centering
\begin{tikzpicture}[roundnode/.style={circle,fill,inner sep=2.5pt},squarednode/.style={rectangle,fill,inner sep=3pt}] 
\node[squarednode,label=below:{$n$}](F1) at (6,-1.5){};
\node[roundnode,label=above:{1}](V1) at (1.5,0){};
\node[roundnode,label=above:{2}](V2) at (3,0){};
\node(V3) at (4.5,0){\ldots};
\node[roundnode,label=above:{$n-1$}](V4) at (6,0){};
\draw[thick, ->] (V4) -- (F1);
\draw[thick, ->] (V1) -- (V2);
\draw[thick, ->] (V2) -- (V3);
\draw[thick, ->] (V3) -- (V4);
\end{tikzpicture}
\caption{The quiver data for the cotangent bundle of the flag variety.} \label{fig1}
\end{figure}

We place a vector space $V_i$ of dimension $i$ at the vertex of the quiver labeled by $i$. We denote the $n$-dimensional framing vector space by $V_n$. We choose the stability condition given by the $G:=\prod_{i=1}^{n-1} GL(V_i)$-character
$$
\theta: G \to \mathbb{C}^{\times}, \qquad (g_i)_i \mapsto \prod_{i=1}^{n-1} (\det g_i)^{-1}
$$
Let
$$
\mu: T^*\left(\prod_{i=1}^{n-1} Hom(V_i,V_{i+1})\right)\to \mathfrak{g}^{*}, \quad \mathfrak{g}=\Lie(G)
$$
be the moment map associated to the natural $G$-action. Points in the associated quiver variety 
$$
X:= \mu^{-1}(0)^{\theta-ss}/G
$$
are represented by tuples of maps
$$
A_i:V_i \to V_{i+1}, \quad B_i: V_{i+1} \to V_i, \quad \text{for } i \in \{1,\ldots,n-1\}
$$
A collection of maps describes a $\theta$-semistable if and only if the maps $A_i$ are injective. The variety $X$ is the cotangent bundle of the full flag variety. 

\subsection{Torus action} 
The action of $(\mathbb{C}^{\times})^n$ on $V_n$ induces an action of $(\mathbb{C}^{\times})^n$ on $X$. The diagonal $\mathbb{C}^{\times} \subset (\mathbb{C}^{\times})^n$ acts trivially on $X$, and we denote by $\bA$ the quotient of $(\mathbb{C}^{\times})^n$ by this subtorus. We denote the coordinates on $(\mathbb{C}^{\times})^n$ by $(u_1,u_2,\ldots,u_n)$ which means that coordinates on $\bA$ are given by $a_i=u_i/u_{i+1}$ for $i \in \{1, \ldots, n-1\}$.

An additional torus $\mathbb{C}^{\times}_{\hbar}$ acts on $X$ by scaling the cotangent data, which is given by the maps $B_i:V_{i+1} \to V_{i}$ for $i\in \{1,\ldots,n-1\}$, by $\hbar^{-1}$. We define 
$$
\bT:=\bA \times \mathbb{C}^{\times}_{\hbar}
$$ 
The torus $\bA \subset \bT$ preserves the symplectic form of $X$ and $\bT/\bA$ scales it with character $\hbar$.

The $\bT$-fixed points of $X$ are indexed naturally by permutations $I=(I_1,\ldots, I_n)$ of $n$. As remarked above, $\theta$-semistability implies that the maps $A_i:V_i \to V_{i+1}$ are injective. We identify each of these vector spaces with a subspace of $V_n=\mathbb{C}^{n}$ and denote by $e_1, \ldots, e_n$ the standard basis of $\mathbb{C}^n$. Then a fixed point indexed by $I$ corresponds to a chain of vector spaces
$$
V_1 \subset V_2 \subset \ldots \subset V_n
$$
such that $V_i/V_{i-1}=\text{Span}_{\mathbb{C}}\{e_{I_i}\}$.

\subsection{}

The vector spaces $V_i$ descend to bundles $\mathcal{V}_i$ on $X$. It is known from \cite{kirv} that the tautological line bundles $\mathscr{L}_i=\det \mathcal{V}_i$ generate $\Pic(X)$. 

We define the K\"ahler torus of $X$ by
$$
\bK= \Pic(X) \otimes_{\mathbb{Z}}\mathbb{C}^{\times}
$$
and write $\zz=(z_1,\ldots,z_{n-1})$ for the coordinates on it induced by the tautological line bundles. The variables $z_i$ are usually referred to as K\"ahler parameters.

The real Lie algebra of $\bK$ is 
$$
\Lie_{\mathbb{R}}(\bK):=\text{cochar}(\bK)\otimes_{\mathbb{Z}} \mathbb{R} = \Pic(X) \otimes_{\mathbb{Z}} \mathbb{R}
$$
We denote elements of $\Lie_{\mathbb{R}}(\bK)$ by $s$ and call them slopes. We identify $\Lie_{\mathbb{R}}(\bK)$ with $\mathbb{R}^{n-1}$ by the map
\begin{equation}\label{lierk}
\sum_{i=1}^{n-1} \mathscr{L}_i\otimes r_i \mapsto (-r_1,\ldots,-r_{n-1}) \in \mathbb{R}^{n-1}
\end{equation}
The minus signs are chosen to match later notation.

\subsection{Dual variety} 
The cotangent bundle of the full flag variety is known to be self dual with respect to 3d mirror symmetry, see \cite{MirSym2}, \cite{msflag}, and \cite{GaKor}. We denote by $X^{!}$ the same variety, constructed as a quiver variety in the same way. For $X^{!}$ we likewise have tori $\bT^{!}$, $\bA^{!}$, and $\bK^{!}$.

\subsection{Polarization and chamber}
Stable envelopes depend on a choice of polarization and chamber. We explain the meaning of these here.

A polarization of $X$ is the choice of a $K$-theory class $T^{1/2}X$ so that the tangent bundle decomposes as
$$
TX= T^{1/2}X+\hbar^{-1} (T^{1/2}X)^{\vee} \in K_{\bT}(X)
$$
We fix the polarization of $X$, given in terms of the tautological bundles $\mathcal{V}_i$, $i=1,\ldots,n$ by
\begin{equation}\label{pol}
T^{1/2}X = \sum_{i=1}^{n-1} \mathcal{V}_i^{\vee} \otimes \mathcal{V}_{i+1} -\sum_{i=1}^{n-1} \mathcal{V}_{i}^{\vee} \otimes \mathcal{V}_{i} \in K_{\bT}(X)
\end{equation}
In terms of the Chern roots $x^{(i)}_1,\ldots x^{(i)}_i$ of $\mathcal{V}_i$, this is given by
$$
T^{1/2}X = \sum_{i=1}^{n-1} \left( \sum_{j=1}^{i} \frac{1}{x^{(i)}_j}\right) \left( \sum_{k=1}^{i+1} x^{(i+1)}_k\right)  - \left( \sum_{j=1}^{i} \frac{1}{x^{(i)}_j}\right) \left( \sum_{k=1}^{i} x^{(i)}_k \right) \in K_{\bT}(X)
$$
At the fixed point given by a permutation $I$ of $\{1,2,\ldots, n\}$, the tangent space can be calculated by substituting $x^{(i)}_j = u_{I_j}$. So 
\begin{align*}
T^{1/2}_I X &= \sum_{i=1}^{n-1} \sum_{j=1}^{i} \sum_{k=1}^{i+1} \frac{u_{I_k}}{u_{I_j}}  -  \sum_{j=1}^{i} \sum_{k=1}^{i} \frac{u_{I_k}}{u_{I_j}}  \\
&= \sum_{1\leq j < k \leq n} \frac{u_{I_k}}{u_{I_j}} \in K_{\bT}(pt)
\end{align*}
and
$$
T_I X =  \sum_{1\leq j < k \leq n} \frac{u_{I_k}}{u_{I_j}} + \hbar^{-1}  \sum_{1\leq j < k \leq n} \frac{u_{I_j}}{u_{I_k}} \in K_{\bT}(pt)
$$

\subsection{Chamber}

A chamber is a choice of connected component of
\begin{equation}\label{comp}
\Lie_{\mathbb{R}}\bA -\bigcup_{w}\{\sigma \in \Lie_{\mathbb{R}}\bA \mid \langle \sigma, w \rangle=0\}
\end{equation}
where the union is taken over all $\bA$-weights of the tangent spaces at the fixed points and $\langle \cdot, \cdot \rangle$ is induced by the natural pairing on characters and cocharacters. A choice of generic cocharacter $\sigma$ of $\bA$ gives a chamber $\mathfrak{C}$, and the dependence of the chamber on the cocharacter is locally constant. The tangent space at a fixed point decomposes into a direct sum of $\bT$-weight spaces
$$
T_I X= \bigoplus_{w\in Hom(\bT,\mathbb{C}^{\times})} V_I(w)
$$
A choice of chamber given by a cocharacter $\sigma$ decomposes the tangent space at a fixed point $I$ into attracting and repelling directions:
$$
T_I X= N^{+}_I+N^{-}_I
$$
where
\begin{align*}
    N_I^{+} &= \bigoplus_{\substack{w \\ \langle \sigma, w \rangle}>0} V_I(w) \\
     N_I^{-} &= \bigoplus_{\substack{w \\ \langle \sigma, w \rangle}<0} V_I(w) 
\end{align*}
Here, $\sigma$ is viewed as a cocharacter of $\bT$ via the inclusion $\bA \subset \bT$. Since the fixed point set is finite and $\sigma$ is chosen to lie in (\ref{comp}), every direction is either attracting or repelling.

In the case of the cotangent bundle of the full flag variety, we fix once and for all the cocharacter
\begin{align}\label{chamb}
\sigma: u \mapsto (u^{-1},u^{-2},\ldots,u^{-n}), \qquad u \in \mathbb{C}^{\times}
\end{align}
and denote the corresponding chamber as $\mathfrak{C}$. With respect to this chamber, attracting weights look like $u_i/u_j$ where $i<j$, or equivalently, like monomials with positive powers in $a_i:=u_i/u_{i+1}$. Explicitly, we have
$$
T_I X= N^{+}_I+N^{-}_I
$$
where
\begin{align}\nonumber \label{normal}
    N_I^- &=  \sum_{\substack{1\leq j < k \leq n \\ I_k > I_j}} \frac{u_{I_k}}{u_{I_j}} + \hbar^{-1}\sum_{\substack{1\leq j < k \leq n \\ I_k < I_j}} \frac{u_{I_j}}{u_{I_k}} \\
    N_I^+ &=  \sum_{\substack{1\leq j < k \leq n \\ I_k < I_j}} \frac{u_{I_k}}{u_{I_j}} + \hbar^{-1}\sum_{\substack{1\leq j < k \leq n \\ I_k > I_j}} \frac{u_{I_j}}{u_{I_k}}
\end{align}

\subsection{Ordering on fixed points}
Given a permutation $I=(I_1,\ldots,I_n)$ of $n$, we define the ordered indices $i^{(k)}_1,\ldots,i^{(k)}_k$ so that
$$
\{i^{(k)}_1 < \ldots < i^{(k)}_k\}= \{I_1,\ldots,I_k\}
$$

\begin{Definition}\label{bruhat}
For permutations $I$ and $J$ with ordered indices $i^{(k)}_m$ and $j^{(k)}_m$, we define
\begin{equation}\label{bruhateq}
I\prec J \iff i^{(k)}_m<j^{(k)}_m \text{   for all   } k=1,\ldots,n-1 \text{  and  } m=1,\ldots,k
\end{equation}
The partial order $\prec$ coincides with the partial order on fixed points given by attraction in section 3.1.2 of \cite{AOElliptic}. In what follows, we will also denote by $\prec$ an arbitrary refinement of this partial order to a total order.
\end{Definition}

\section{3d mirror symmetry}\label{mirsym}
\subsection{Exchange of equivariant and K\"ahler parameters}
One property of 3d mirror symmetry, see \cite{KS2} and \cite{dinksmir}, is the existence of a bijection
$$
X^{\bT} \longleftrightarrow {X^{!}}^{\bT^{!}}
$$
and an isomorphism of tori
$$
\kappa: \bT \times \bK \times \mathbb{C}^{\times}_q \to \bT^{!} \times \bK^{!} \times \mathbb{C}^{\times}_q
$$
that induces isomorphisms $\bT \cong \bK^{!}$ and $\bK \cong \bT^{!}$. The torus $\mathbb{C}^{\times}_q$ acts on the domain of quasimaps to $X$ and $X^{!}$, as will be explained in section \ref{quasimaps} below.

Following \cite{msflag}, we define the bijection on fixed points by
$$
I \longleftrightarrow I^{-1} =:I^{!}
$$
and the map $\kappa$ by
\begin{align}\label{kap} \nonumber
    z_{i} &\mapsto \hbar^{!} a^{!}_i \\ \nonumber
    a_i &\mapsto \frac{\hbar^{!}}{q} z^{!}_i \\ \nonumber
    \hbar &\mapsto \frac{q}{\hbar^{!}} \\
    q &\mapsto q
\end{align}
The differential $d\kappa$ induces isomorphisms 
\begin{equation}\label{lielie}
\Lie_{\mathbb{R}}(\bA) \cong \Lie_{\mathbb{R}}(\bK^{!}), \qquad \Lie_{\mathbb{R}}(\bK) \cong \Lie_{\mathbb{R}}(\bA^{!})
\end{equation}
We also obtain an induced map
$$
K_{\bT\times \mathbb{C}^{\times}_q}(pt)_{loc}[[\zz]] \cong \mathbb{C}(\aa,\hbar,q)[[\zz]] \to \mathbb{C}(\aa^{!},\hbar^{!},q)[[\zz^{!}]]\cong K_{\bT^{!}\times \mathbb{C}^{\times}_q}(pt)_{loc}[[\zz^{!}]]
$$
where we use the subscript $loc$ to denote the localized $K$-theory. We will abuse notation and also denote this last map by $\kappa$.

\subsection{Walls and resonances}
The $K$-theoretic stable envelopes depend on a choice of slope $s\in \Lie_{\mathbb{R}}(\bK)$. The dependence is locally constant, and the $K$-theoretic stable envelopes change only when $s$ crosses the walls of a certain hyperplane arrangement, which we denote by $\Wall(X)$.

From (\ref{lierk}) and (\ref{lielie}) we have 
$$
s=(s_1,\ldots,s_{n-1}) \in \mathbb{R}^{n-1}\cong \Lie_{\mathbb{R}}(\bK)\cong \Lie_{\mathbb{R}}(\bA^{!}) 
$$
and we write
\begin{align*}
    \zz q^{s} &= (z_1 q^{s_1},\ldots, z_{n-1} q^{s_{n-1}}) \\
 \aa^{!} q^{s} &=(a_1^{!} q^{s_1},\ldots, a_{n-1}^{!} q^{s_{n-1}})
\end{align*}


To each element $\mathsf{w} \in \Lie_{\mathbb{R}}(\bA)$ with $\mathsf{w}=(\mathsf{w}_1,\ldots,\mathsf{w}_{n-1})$, we let $\nu_{\mathsf{w}}$ be the cyclic subgroup of $\bA$ generated by
\begin{equation}\label{resgroup}
(e^{2 \pi i w_1},\ldots, e^{2 \pi i w_{n-1}}) \in \bA
\end{equation}
We define
$$
\Res(X)=\{\mathsf{w} \in \Lie_{\mathbb{R}}(\bA) \, \mid \, X^{\nu_{\mathsf{w}}} \neq X^{\bA} \}
$$
In \cite{KS1}, it is shown that 
$$
\Res(X)=\{\mathsf{w} \, \mid \, \langle \alpha, \mathsf{w} \rangle+m=0 \, \, \text{for some} \, \, m \in \mathbb{Z},  \, \,  I\in X^{\bA}, \, \,  \alpha \in \text{char}_{\bA}(T_I X) \}
$$

Under the identification (\ref{lielie}), the walls and resonances are exchanged (\cite{KS2} Theorem 2):
$$
\Wall(X) \cong \Res(X^{!}), \quad \Res(X) \cong \Wall(X^{!})
$$
A simple calculation shows that under (\ref{lierk}), the walls are given by
$$
\Wall(X)=\{(s_1,\ldots,s_{n-1}) \, \mid \, \exists I \subset \{1,\ldots,n-1\}, \, \sum_{i \in I} s_i \in \mathbb{Z}\}
$$

We use the term generic slopes to refer to elements of $\Lie_{\mathbb{R}}(\bK)\setminus \Wall(X)$.

\section{Index vertex}\label{indsection}
In this section, we define one of the objects involved in our main theorem: the index vertex. We are interested in it for $X^{!}$.

\subsection{Quasimaps}\label{quasimaps}
Let $[x:y]$ denote homogeneous coordinates on $\mathbb{P}^1$. We denote
$$
0=[0:1], \quad \infty=[1:0]
$$
The torus $\mathbb{C}^{\times}_q$ acts on $\mathbb{P}^1$ by
\begin{equation}\label{Cq}
q \cdot [x_0:x_1] = [x_0 q:x_1], \quad q \in \mathbb{C}^{\times}
\end{equation}

For $p\in (X^{!})^{\bT^{!}}$, let $\qm_{p}$ be the moduli space of stable quasimaps from $\mathbb{P}^1$ to $X^{!}$ that take the value $p$ at $\infty$, see \cite{qm} and \cite{pcmilect} section 4. The data of a stable quasimap to $X^{!}$ provides:
\begin{itemize}
    \item Vector bundles $\mathscr{V}_i$ for $i\in \{1,\ldots,n-1\}$ over $\mathbb{P}^1$ such that $\text{rank}(\mathscr{V}_i)=i$.
    \item A trivial vector bundle $\mathscr{V}_n$ over $\mathbb{P}^1$ of rank $n$.
    \item  A section
$$
f \in H^0(\mathbb{P}^1,\mathscr{M}\oplus  (\hbar^{!})^{-1}(\mathscr{M})^{\vee}),\quad f(\infty)=p
$$
where 
$$
\mathscr{M}= \bigoplus_{i=1}^{n-1} Hom(\mathscr{V}_i,\mathscr{V}_{i+1})
$$
\end{itemize}
such that the section $f$ lands in the GIT stable locus for all but finitely many points of $\mathbb{P}^1$. We abuse notation by writing a quasimap as $f$, with the understanding that the data of the vector bundles are included.

The degree of a quasimap is given by 
$$
\deg f=(\deg \mathscr{V}_1,\ldots, \deg \mathscr{V}_{n-1})
$$
which gives a decomposition
$$
\qm_p = \bigsqcup_{\dd} \qm^{\dd}_p
$$
of the quasimap moduli space into components corresponding to quasimaps with fixed degree $\dd$. It is known that the degrees for which $\qm^d_p$ is nonempty lie inside a certain cone, see \cite{pcmilect} section 7.2.

The actions of $\bT^{!}$ on $X^{!}$ and of $\mathbb{C}^{\times}_q$ on $\mathbb{P}^1$ induce an action on quasimaps. We denote $\bT^{!}_q=\bT^{!}\times\mathbb{C}^{\times}_{q}$.

\subsection{Virtual classes on $\qm_p$}
It is known that $\qm_p$ has a perfect obstruction theory, which allows one to define the symmetrized virtual structure sheaf $\vrs$ and virtual tangent space $\mathscr{T}_{\text{vir}}$, see \cite{qm} Theorem 7.2.2 and \cite{pcmilect} section 6.

Fix $f \in \left(\qm_{p}\right)^{\bT^{!}_q}$. As a $\bT^{!}_q$-module, the virtual tangent space of $\qm_p$ at $f$ is
$$
\mathscr{T}_{\textrm{vir},f}=H^*\left(\mathbb{P}^1, \mathcal{T}^{1/2}+ \hbar^{-1}(\mathcal{T}^{1/2})^{\vee}\right)
$$
where 
\begin{equation}\label{Tbundle}
\mathcal{T}^{1/2}=\bigoplus_{i=1}^{n-1} Hom(\mathscr{V}_i,\mathscr{V}_{i+1}) - \bigoplus_{i=1}^{n-1} End(\mathscr{V}_i)
\end{equation}

The reduced virtual tangent space at $f$ is defined to be
$$
\mathscr{T}_{\textrm{vir},f}^{\textrm{red}}=\mathscr{T}_{\textrm{vir},f}-T_p X^{!}
$$
The class $\mathcal{T}^{1/2}$ induces a polarization $\mathscr{T}^{1/2}$ of the reduced virtual tangent space at a fixed quasimap $f$:
\begin{equation}\label{qmpol}
\mathscr{T}_{\mathrm{vir},f}^{\mathrm{red}}=\mathscr{T}^{1/2}+\hbar^{-1}\left(\mathscr{T}^{1/2}\right)^{\vee} \quad \text{where} \quad \mathscr{T}^{1/2}=H^*(\mathbb{P}^1, \mathcal{T}^{1/2})-T^{1/2}_{p}X^{!}
\end{equation}
\subsection{Index vertex}
Let $s\in \Lie_{\mathbb{R}}(\bA^{!})\setminus \Res(X^{!})$ and denote the substitution $a_i=q^{s_i}$ by $\aa=q^{s}$. Such an $s$ induces a decomposition 
$$
\mathscr{T}^{1/2}=\mathscr{T}^{1/2}_{s,+}+\mathscr{T}^{1/2}_{s,-}
$$
where $\mathscr{T}^{1/2}_{s,\pm}$ consists of the terms of $\mathscr{T}^{1/2}$ that, after the substitution $\aa=q^s$, have finite limit as $q^{\pm} \to \infty$.

\begin{Definition}\label{sindex}
Let $s\in \Lie_{\mathbb{R}}(\bA^{!})\setminus \Res(X^{!})$. The $s$-index of $f\in \left(\qm_p\right)^{\bT^{!}_q}$ is
$$
\mathcal{I}_{s}(f)=\text{rank}\left( \mathscr{T}^{1/2}_{s,-}\right)
$$
where we understand the rank of a virtual $\bT^{!}_q$-bundle to be counted with sign.\footnote{In other words, if $A$ and $B$ are $\bT^{!}_q$-modules, then $\text{rk}(A-B)=\text{rank}(A)-\text{rank}(B)$.}
\end{Definition}

\begin{Definition}[\cite{pcmilect} section 7.3]\label{ind}
The index vertex of $X^{!}$ at $p \in (X^{!})^{\bT^{!}}$ with respect to $s \in \Lie_{\mathbb{R}}(\bA^{!})\setminus \Res(X^{!})$ is the generating function
$$
\Ind_p^{!s}=\sum_{\substack{\dd \\ \qm^{d}_p \neq \emptyset}} \left(\zz^{!}\right)^{-\dd} \sum_{f \in (\qm^{\dd}_p)^{\bT^{!}_q}} \left( \frac{\hbar^{!}}{q}\right)^{\mathcal{I}_s(f)}
$$
where $\left(\zz^{!}\right)^{-\dd}=\prod_{i=1}^{n-1} \left(z_i^{!}\right)^{-d_i}$.
\end{Definition}
The choice of $-\dd$ in the definition of the index vertex is made for the sake of consistency with our other conventions from \cite{msflag}, which will be used below.



\subsection{Vertex function}
There is an alternative perspective on the index vertex which describes it as a limit of the so-called vertex function for $X^{!}$. Since we will need the vertex function for both $X$ and $X^{!}$, we briefly switch our variables back to those of $X$. The vertex function for $X^{!}$ is given by the same formula as below after the trivial change of variables
$$
\hbar \to \hbar^{!}, \quad u_i \to u_i^{!}, \quad \text{and} \quad z_i \to z_i^{!}
$$
\begin{Definition}\label{degrees}
We define $C\subset \mathbb{Z}\times \mathbb{Z}^2\times \ldots \times \mathbb{Z}^{n-1}$ as the collection of integers $d_{i,j}$ where $i\in \{1,\ldots,n-1\}$ and $j\in\{1,\ldots,i\}$ such that
\begin{itemize}
    \item $d_{i,j}\geq 0$ for all $i,j$.
    \item For each $i\in\{1,\ldots,n-2\}$, there exists $\{j_1,\ldots,j_i\}\subset \{1,\ldots,i+1\}$ so that $d_{i,k}\geq d_{i+1,j_k}$ for all $k$.
\end{itemize}
\end{Definition}

The vertex function of $X$ is defined by an equivariant count of quasimaps from $\mathbb{P}^1$ to $X$, see Theorem \ref{qmver} below. To streamline our presentation here, we define the vertex function through a formula.

\begin{Definition}\label{ver}
The vertex function of the cotangent bundle of the full flag variety restricted to a fixed point $I$ is given by the following power series:
\begin{multline*}
V_I(\aa,\zz) = \sum_{d_{i,j} \in C} \zz^{\dd} \prod_{i=1}^{n-2} \prod_{j=1}^i \prod_{k=1}^{i+1}\frac{\left(\hbar \frac{u_{I_k}}{u_{I_j}}\right)_{d_{i,j}-d_{i+1,k}}}{\left(q \frac{u_{I_k}}{u_{I_j}}\right)_{d_{i,j}-d_{i+1,k}}} \\ \prod_{i=1}^{n-1}  \prod_{j,k=1}^{i}\frac{\left(q \frac{u_{I_k}}{u_{I_j}}\right)_{d_{i,j}-d_{i,k}}}{\left(\hbar \frac{u_{I_k}}{u_{I_j}}\right)_{d_{i,j}-d_{i,k}}} 
\prod_{i=1}^n \prod_{j=1}^{n-1} \frac{\left(\hbar \frac{u_i}{u_{I_j}} \right)_{d_{n-1,j}}}{\left(q \frac{u_i}{u_{I_j}} \right)_{d_{n-1,j}}}
\end{multline*}
where $\zz^{\dd}=\prod_{i=1}^{n-1}\prod_{j=1}^i z_i^{d_{i,j}}$ and $(x)_d$ denotes the $q$-Pochammer symbol
$$
(x)_d:=\frac{\varphi(x)}{\varphi(x q^d)}, \qquad \varphi(x)=\prod_{i=0}^{\infty} (1-x q^i)
$$
\end{Definition}
In what follows, we will need the roof function, which is defined by
\begin{equation}\label{roof}
\hat{a}(t)=\frac{1}{t^{1/2}-t^{-1/2}}, \quad \hat{a}(t_1+t_2-t_3)=\frac{\hat{a}(t_1)\hat{a}(t_2)}{\hat{a}(t_3)}
\end{equation}

The geometric meaning of the vertex function is given by the following theorem. For precise definitions, see \cite{pcmilect} section 7.
\begin{Theorem}[\cite{KorZeit} Theorem 3.1]\label{qmver}
Let $\vrs^{\dd}$ be the symmetrized virtual structure sheaf on $\qm^{d}_I$. The vertex function $V_I(\aa,\zz)$ in Definition \ref{ver} is equal to
\begin{equation}\label{qmver2}
\frac{1}{\hat{a}(T_I X)}\sum_{\substack{\dd \\  \qm^{\dd}_{I} \neq \emptyset}} \left(\zz^{\#}\right)^{-\dd} \chi\left(\vrs^{\dd}\right) \in K_{\bT_q}(pt)_{loc}[[\zz]]
\end{equation}
where $\chi$ denotes the $K$-theoretic equivariant Euler characteristic and
$$
\left(\zz^{\#}\right)^{-\dd}= \prod_{i=1}^{n-1} \left(z_i^{\#}\right)^{-d_i} \quad \text{where} \quad  z^{\#}_i=\hbar^{a_i/2}q^{b_i/2} z_i 
$$
for some $a_i, b_i \in \mathbb{Z}$.
\end{Theorem}
\begin{Remark}
The $\hat{a}$ prefactor is equivalent to replacing the virtual tangent space by the reduced virtual tangent space in localization formulas, which causes the series to start with 1.
\end{Remark}
\begin{Remark}
The precise form of the shift $z^{\#}_i=\hbar^{a_i/2}q^{b_i/2} z_i$ is described in the proof of Proposition \ref{2index} below.
\end{Remark}

\subsection{Index limit}
In this subsection, we switch back to the variables for $X^{!}$.


\begin{Definition}\label{indlimit}
The index limit of an element of $F(\aa^{!},\zz^{!})\in K_{\bT^{!}_q}(pt)[[\zz^{!}]]$ with respect to $s\in \Lie_{\mathbb{R}}(\bA^{!})$ is defined to be 
$$
\kappa\left(\lim_{q\to 0} \kappa^{-1}\left(F(\aa^{!} , \zz^{!})\right)\big|_{\zz=\zz q^{s+1}} \right)
$$
provided this limit exists. Here, $\zz q^{s+1}=(z_1 q^{s_1+1}, \ldots, z_{n-1} q^{s_n+1})$.
\end{Definition}

\begin{Remark}
The reason we shift by $q^{s+1}$ instead of $q^s$ is for consistency with the index vertex, see Proposition \ref{2index} below.
\end{Remark}

\begin{Remark}
It may appear more natural to consider the limit
$$
\lim_{q\to 0} V^{!}_{I^{!}}(\aa^{!} q^{s},\zz^{!})
$$
However, this limit will not exist for vertex functions. For example, the vertex function of one of the fixed points of $X^{!}=T^*\mathbb{P}^1$ is
$$
V(\aa^{!},\zz^{!})=\sum_{d=0}^{\infty} \frac{(\hbar^{!})_{d} (\hbar^{!} u_1^{!}/u_2^{!} )_d}{(q)_d (q u_1^{!}/u_2^{!} )_d} {z^{!}_1}^{d}
$$
and one can easily check that the limit
$$
\lim_{q\to 0} V(\aa^{!} q^{s},\zz^{!})
$$
does not exist. On the other hand, the index limit does exist. By Proposition \ref{limver} below, this is also the case in general.
\end{Remark}
\begin{Remark}\label{rm2}
The index limit with respect to $s$ could also be defined by formally substituting $\hbar^{!}=q/t$, shifting $\aa^{!} \to \aa^{!} q^{s}$, taking the limit $q\to 0$, and substituting back $t=q/\hbar^{!}$.
\end{Remark}

\begin{Proposition}\label{limver}
The index limit of the vertex function restricted to any fixed point exists for all $s$.
\end{Proposition}
\begin{proof}
The coefficients of the vertex function consist of terms of the form
$$
\left(\frac{(\hbar^{!} x)_d}{(q x)_d}\right)^{\pm1}
$$
where $x$ is a character of $\bA^{!}$ and $d\in \mathbb{Z}$. Using Remark \ref{rm2}, one can see that the limit of all such terms exists.
\end{proof}


\begin{Proposition}\label{2index}
For $s\in \Lie_{\mathbb{R}}(\bA^{!})\setminus \Res(X^{!})$, the index limit with respect to $s$ of the vertex function of $X^{!}$ restricted to $I^{!}$ is equal to the index vertex at $I^{!}$:
$$
\kappa\left( \lim_{q\to 0} \kappa^{-1}\left(V^{!}_{I^{!}}(\aa^{!},\zz^{!}) \right)\big|_{\zz \to \zz q^{s+1}} \right)= \Ind^{!s}_{I^{!}}
$$
\end{Proposition}
\begin{proof}

\begingroup
\renewcommand*{\arraystretch}{2}
\begin{table}[b]
\centering
\begin{tabular}{|c|c|}
\hline
       Virtual sub-bundle of $\mathcal{T}^{1/2}$  & Contribution to $\mathscr{T}^{\text{red}}_{\text{vir},f}$  \\ \hline
      $ \pm x q^{d} \mathcal{O}(d) \pm (\hbar^{!})^{-1}x^{-1} q^{-d}\mathcal{O}(-d)$   & $\pm x q \left(1+q+\ldots + q^{d-1}\right) \mp \dfrac{1}{x \hbar^{!}}\left( 1 + q^{-1} + \ldots q^{-d+1} \right)$ \\ \hline 
     $\pm x q^{-d} \mathcal{O}(-d)\pm (\hbar^{!})^{-1} x^{-1} q^{d} \mathcal{O}(d)$ & $\pm \dfrac{q }{x \hbar^{!}}\left(1+q+\ldots + q^{d-1}\right) \mp x\left( 1 + q^{-1} + \ldots q^{-d+1} \right)$ \\ \hline
\end{tabular}
\caption{Contributions of $\mathcal{T}^{1/2}$ to $\mathscr{T}^{\text{red}}_{\text{vir},f}$.}\label{tab1}
\end{table}
\endgroup

Up to normalization, the vertex function of $X^{!}$ is equal to the generating function of the equivariant Euler characteristic of $\qm^{\dd}_{I^{!}}$ as in Theorem \ref{qmver}. We will compute this here using equivariant localization.

At a fixed point $f\in \left(\qm^{\dd}_{I^{!}}\right)^{\bT^{!}_q}$, the vector bundles $\mathscr{V}_i$ over $\mathbb{P}^1$ split into the sum of equivariant line bundles: 
$$
\mathscr{V}_i=\bigoplus_{j=1}^i w_{i,j} q^{d_{i,j}}\mathcal{O}(d_{i,j})
$$
where $w_{i,j}$ stands for a character of $\bA^{!}$, see section 4.5 of \cite{Pushk1}. The injectivity of the maps $V_{i-1}\to V_{i}$ constrains the posible degrees of the bundles $\mathscr{V}_i$. In particular, the collection $\{-d_{i,j}\}$ of negatives of the degrees must lie in $C$. Hence we have
$$
\mathcal{T}^{1/2} = \bigoplus_{i} x_{i}q^{a_i} \mathcal{O}(a_i)-\bigoplus_{j} y_{j}q^{b_j} \mathcal{O}(b_j)
$$
where $x_i$ and $y_j$ are characters of $\bA^{!}$, $a_i$ and $b_j$ are integers, and the sums are taken over some indexing sets.

Each of the possible virtual sub-bundles $\pm x q^{\pm d} \mathcal{O}(\pm d)$ of $\mathcal{T}^{1/2}$ comes with a pair $\pm (\hbar^{!})^{-1} x^{-1} q^{\mp d} \mathcal{O}(\mp d)$ which together contribute to the reduced virtual tangent space as in Table \ref{tab1} (see \cite{Pushk1} Lemma 1). We assume in Table \ref{tab1} without loss of generality that $d\geq 0$.

By definition of the symmetrized virtual structure sheaf, these terms contribute to localization formula via the roof function (\ref{roof}). The contributions of each of these terms is given in Table \ref{tab2}.

\begingroup
\renewcommand*{\arraystretch}{2}
\begin{table}[h]
\centering
\begin{tabular}{|c|c|} \hline
       Virtual sub-bundle of $\mathcal{T}^{1/2}$  & Contribution to (\ref{qmver2}) \\ \hline 
      $ \pm x q^{d} \mathcal{O}(d) \pm (\hbar^{!})^{-1}x^{-1} q^{-d}\mathcal{O}(-d)$   & $\left((-q^{1/2} (\hbar^{!})^{-1/2})^d \dfrac{(\hbar^{!} x)_d}{(q x)_d}\right)^{\pm 1}$ \\  \hline 
     $\pm x q^{-d} \mathcal{O}(-d)\pm (\hbar^{!})^{-1} x^{-1} q^{d} \mathcal{O}(d)$ & $\left((-q^{1/2}(\hbar^{!})^{-1/2})^{-d}\dfrac{(\hbar^{!} x)_{-d}}{(q x)_{-d}}\right)^{\pm 1}$ \\ \hline
\end{tabular}
\caption{Contributions in the localization formula.}\label{tab2}
\end{table}
\endgroup


The terms $(-q^{1/2}\hbar^{-1/2})^{\pm d}$ account for precisely the difference between $z_i$ and $z^{\#}_i$ in Theorem \ref{qmver}.

Applying Remark \ref{rm2}, we see that the contributions to the index limit arise from monomials in 
$$
 \mathscr{T}^{1/2}\big|_{\aa^{!}= q^{s}}
$$
that tend to $0$ as $q\to \infty$. Each such monomial contributes a power of $\left(\hbar^{!}/q\right)^{\pm 1}$, with the sign determined by the sign of the monomial. Putting all this together, we see that the index limit of the vertex function is equal to the index vertex.
\end{proof}

We can now extend the definition of the index vertex to include non-generic slopes.

\begin{Definition}
The index vertex of $X^{!}$ at $I^{!}$ with respect to $s\in \Lie_{\mathbb{R}}(\bA^{!})$ is defined to be the index limit of $V^{!}_{I^{!}}(\aa^{!},\zz^{!})$ with respect to $s$.
\end{Definition}
Thanks to Proposition \ref{2index}, this definition agrees with Definition \ref{ind} for $s \in \Lie_{\mathbb{R}}(\bA^{!})\setminus \Res(X^{!})$ and extends it for $s \in \Res(X^{!})$.

\section{Stable envelopes}\label{stabsection}

In this section, we explain our conventions used for stable envelopes. The main references for stable envelopes are \cite{AOElliptic} and section 9 of \cite{pcmilect}. 

\subsection{Notations}
Given a $\bT$-module expressed in weights as $V=w_1+\ldots + w_r \in K_{\bT}(pt)$ with $w_i\neq 1$ for all $i$, we define the symmetric and exterior powers as
$$
\sympow(V):= \prod_{i=1}^{r} (1-w_i)^{-1}, \quad \extpow V :=\prod_{i=1}^{r} (1-w_i)
$$
We extend these to all of $K_{\bT}(pt)$ by
$$
\sympow(-V):= \extpow V, \quad \extpow (-V):=\sympow(V)
$$
Similarly, we define
$$
\Theta(V)= \prod_{i=1}^{r} \vartheta(w_i) \quad \text{and} \quad \Phi(V)= \prod_{i=1}^{r} \varphi(w_i)
$$
where
$$
\varphi(x)=\prod_{i=0}^{\infty} (1-x q^i) \quad \text{and} \quad \vartheta(x)=(x^{1/2}-x^{-1/2}) \varphi(q x) \varphi(q/x)
$$
We similarly extend these to $K_{\bT}(pt)$ by multiplicativity.

\subsection{Stable envelopes}
Stable envelopes depend on a choice of polarization and chamber. In all that follows, we assume that these are given by (\ref{pol}) and (\ref{chamb}).

For $I\in X^{\bA}$, let $\stab^{X,Ell}_{-\mathfrak{C}}(I)$ be the elliptic stable envelopes of $I$ for $X$ corresponding to the polarization $T^{1/2}X$ and chamber $-\mathfrak{C}$. 

The elliptic stable envelope of a fixed point gives a section of a line bundle over the extended elliptic cohomology scheme of $X$. This scheme can be described as
$$
\text{Ell}_{\bT}(X)=\left(\bigsqcup_{I \in X^{\bT}} \widehat{O}_I\right)/\Delta
$$
where $\widehat{O}_{I}$ is a product of elliptic curves isomorphic to $\mathbb{C}^{\times}/q^{\mathbb{Z}}$ for fixed $q$ with $|q|<1$ and $\Delta$ denotes a certain gluing of these abelian varieties, see \cite{SmirnovElliptic} section 2.13. Restricting this section to a component $\widehat{O}_J$, one obtains the matrix of restrictions of the elliptic stable envelope:
$$
T^X_{I,J}:=\stab^{X,Ell}_{-\mathfrak{C}}(I)\big|_{\widehat{O}_J}
$$
We use the normalization of the elliptic stable envelope determined by
$$
T^{X}_{I,I} = \Theta(N_I^{+})
$$
which differs from the normalization of \cite{AOElliptic} by a sign.
We assume the fixed points are ordered from highest to lowest with respect to $\prec$ from Definition \ref{bruhat}, which means that the matrix of restrictions is upper triangular. Explicit formulas for the elliptic stable envelopes in terms of the theta function $\vartheta$ can be written using the so-called elliptic weight functions, see \cite{RTV} and section 3 of \cite{msflag}. More generally, explicit formulas for the elliptic stable envelopes of any type $A$ quiver variety were written in \cite{dinkinselliptic}.

We also define a normalized matrix of restrictions of the elliptic stable envelope by
$$
\widetilde{T}^{X}_{I,J}= \frac{\stab^{X, Ell}_{-\mathfrak{C}}(I)\big|_{\widehat{O}_J}}{\stab^{X,Ell}_{-\mathfrak{C}}(J)\big|_{\widehat{O}_J}}
$$

Similarly, we have $K$-theoretic stable envelopes $\stab^{s,X,K}_{-\mathfrak{C}}(I)$ that depend further on a choice of slope $s\in \Lie_{\mathbb{R}}(\bK)$, see \cite{pcmilect} section 9 and \cite{OS} section 2. We normalize them by requiring that the diagonal terms of the matrix of fixed point restrictions are given by
$$
\stab^{s,X,K}_{-\mathfrak{C}}(I)\big|_{I} = \sqrt{\frac{\det N_{I}^{+}}{\det T^{1/2}_I X}} \extpow (N_{I}^{+})^{\vee}
$$
We denote the matrix of restrictions of the $K$-theoretic stable envelopes by 
$$
A^{s,X}_{I,J}= \stab^{s,X,K}_{-\mathfrak{C}}(I)\big|_J
$$
and a renormalized restriction matrix by 
$$
\widetilde{A}^{s,X}_{I,J}= \frac{\stab^{s,X,K}_{-\mathfrak{C}}(I)\big|_J}{\stab^{s,X,K}_{-\mathfrak{C}}(J)\big|_J}
$$

\section{Index vertex and stable envelopes}\label{theoremsection}

This section contains our main theorems, which relate the index vertex of $X^{!}$ to the $K$-theoretic stable envelopes of $X$.

\subsection{Big enough slopes}
In what follows, we will also be interested in the index limits of $\kappa\left(V_{I}(\aa,\zz)\right)$ and $\Phi((q-\hbar^{!})N_{I^{!}}^{!+})$ in the sense of Definition \ref{indlimit}. It will be convenient to have a notion of slopes for which these limits are trivial.

\begin{Definition}\label{ample}
A slope $s\in \Lie_{\mathbb{R}}(\bA^{!})$ is said to be big enough if $s_i>0$ for all $i$ and
$$
s_i+s_{i+1} + \ldots + s_{j} > j-i-1
$$
for all $j\geq i$.
\end{Definition}

\begin{Lemma}\label{ample2}
If $s\in \Lie_{\mathbb{R}}(\bA^{!})$ big enough, then
$$
\lim_{q\to 0} V_{I}(\aa,\zz q^{s+1})=1
$$
and
$$
\kappa\left( \lim_{q\to0} \kappa^{-1}\left(\Phi((q-\hbar^{!}) N_{I^{!}}^{!+})\right)\big|_{\zz\to\zz q^{s+1}} \right)=1
$$
for all permutations $I$.
\end{Lemma}
\begin{proof}
The follows from Definition \ref{ver} and (\ref{normal}) by a straightforward computation.
\end{proof}

\subsection{Generic slopes}
Let $\mathsf{D}$ be the diagonal matrix given by 
\begin{equation}\label{Dmat}
\mathsf{D}=\text{diag}\left(\sqrt{\frac{\det T_I^{1/2}}{\det N_I^+}}\right)_{I\in X^{\bT}}
\end{equation}
Let $\sympow\left( \hbar \otimes N^+ \right)$ be the column vector
$$
\sympow\left( \hbar \otimes N^+ \right)=\left( \sympow\left( \hbar \otimes N_I^+ \right)\right)_{I \in X^{\bT}}
$$
and $\Ind^{!s}$ be the column vector
$$
\Ind^{!s}=\left(\Ind^{!s}_{I^{!}}\right)_{I \in X^{\bT}}
$$
of index vertices for $X^{!}$. We remind the reader that $I^{!}$ is the inverse permutation of $I$. 


Our main theorem describes the index vertex of $X^{!}$ for generic slopes in terms of the $K$-theoretic stable envelopes of $X$.
\begin{Theorem}\label{mainthm}
If $s \in \Lie_{\mathbb{R}}(\bK)\setminus \Wall(X)$ is big enough, then
$$
\Ind^{!s} = \kappa\left( \mathsf{D} \cdot \widetilde{A}^{s+1,X} \cdot \mathsf{D}^{-1} \cdot \sympow\left(\hbar \otimes N^+\right) \right)
$$
Equivalently,
$$
\Ind^{!s}_{I^{!}}=\kappa\left( \sum_{J \in X^{\bT}} \sqrt{\frac{\det T^{1/2}_I X}{ \det N_I^+}}  \widetilde{A}^{s+1,X}_{I,J} \sqrt{\frac{ \det N_J^{+}}{\det T_J^{1/2}X}}\sympow\left( \hbar \otimes N_J^{+} \right) \right)
$$
\end{Theorem}
\begin{proof}
In \cite{msflag}, we proved that 
\begin{multline}\label{msflags}
\Phi((q-\hbar^{!}) N_{I^{!}}^{! +}) V^{!}_{I^{!}}(\aa^{!},\zz^{!}) \\ =\kappa\left(\sum_{J \in X^{\bT}} \sqrt{\frac{\det T^{1/2}_I X }{\det N_I^+}} \widetilde{T}^{X}_{I,J}   \sqrt{\frac{\det N_J^{+}}{\det T_J^{1/2}X}} \Phi((q-\hbar) N_{I}^{+}) V_{J}(\aa,\zz) \right)
\end{multline}
To deduce Theorem \ref{mainthm}, we take the index limit of both sides with respect to $s$. 

Lemma \ref{ample2} implies that $\Phi((q-\hbar^{!}) N_{I^{!}}^{!+})$ and $V_J(\aa,\zz)$ contribute a factor of $1$ to the limit. The term
$$
 \sqrt{\frac{\det T^{1/2}_I X }{\det N_I^+}}  \sqrt{\frac{\det N_J^{+}}{\det T_J^{1/2}X}}
$$
does not depend on $\zz$ or $q$. The limit of the normalized stable envelope $\widetilde{T}^{X}_{I,J}$ was calculated in \cite{KS2}. For $s \notin \Wall(X)$, we have 
$$
\lim_{q\to 0} \widetilde{T}^{X}_{I,J}|_{z=zq^{s+1}} = \widetilde{A}^{s+1,X}_{I,J}
$$
The contribution from $\Phi((q-\hbar)N_J^{+})$ is given by
$$
\prod_{w \in \text{char}_{\bT}(N_J^{+})} (1-\hbar w)^{-1}= \sympow\left( \hbar \otimes N_J^{+} \right)
$$
Putting all this together gives the result.
\end{proof}

The normalized $K$-theoretic stable envelopes are related to the usual ones by 
$$
A^{s+1,X}_{I,J} = \sqrt{\frac{\det N_J^+}{\det T^{1/2}_J X}} \extpow (N_J^+)^{\vee} \widetilde{A}^{s+1,X}_{I,J}
$$
Also,
$$
N_{J}^{+}= \hbar^{-1}  \left( N_{J}^{-}\right)^{\vee} \implies \sympow(\hbar \otimes N_J^+ ) = 
\sympow\left( \left( N_J^{-}\right)^{\vee} \right)
$$
So Theorem \ref{mainthm} is equivalent to
$$
\Ind^{!s}_{I^{!}}=\kappa\left( \sum_{J \in X^{\bT}} \sqrt{\frac{\det T^{1/2}_I X}{ \det N_I^+}}  A^{s+1,X}_{I,J} \sympow\left( T_JX^{\vee} \right) \right)
$$

By equivariant localization, the right side of this equation is the $K$-theoretic equivariant Euler characteristic of the $K$-theoretic stable envelope of $I$ twisted by a line bundle:
\begin{Theorem}
$$
\Ind^{!s}_{I^{!}}= \kappa \left(\chi\left(\sqrt{\frac{\det T^{1/2}_I X}{ \det N_I^+}}  \stab^{s+1,X,K}_{-\mathfrak{C}}(I) \right) \right)
$$

\end{Theorem}

As a consequence, we obtain:
\begin{Corollary}
For big enough generic slopes, the index vertex $\Ind^{! s}_{I^{!}}(\aa,\zz)$ is a rational function of $\aa^{!}$, $\zz^{!}$, and $\sqrt{\hbar^{!}}$.
\end{Corollary}

\subsection{Non-generic slopes}
The limits of the elliptic stable envelopes for non-generic slopes is one of the main results of \cite{KS2}. 
Under the identification given by $\kappa$, a slope $s \in \Lie_{\mathbb{R}}(\bK)$ can be viewed as an element of $\Lie_{\mathbb{R}}(\bA^{!})$, which has an associated cyclic subgroup $\nu_{s}\subset \bA^{!}$ as in (\ref{resgroup}). In particular, if $s \in \Wall(X)=\Res(X^{!})$, then we obtain a subvariety
$$
Y_s := (X^{!})^{\nu_{s}} \neq (X^{!})^{\bA^{!}}
$$
The limits of the elliptic stable envelopes for non-generic slopes is expressed in terms of the $K$-theoretic stable envelopes of $Y_s$. We need to explain the choice of chamber, polarization, and slope used in the latter.


We first explain the choice of slope. Let $\mathscr{U}_0$ denote an open analytic neighborhood of $0$ in $\Lie_{\mathbb{R}}(\bK^{!})$. Let $\Wall_0(X^{!})$ be the set of walls passing through $0$. Then $\mathscr{U}_0 \setminus \Wall_0(X^{!})$ is the disjoint union of connected components. Let $\mathfrak{D}_+(X^{!})$ denote the connected component containing ample line bundles on $X^{!}$. Explicitly, $\mathfrak{D}_+(X^{!})$ is generated by small positive real multiples of the Chern classes of the tautological line bundles $\mathscr{L}^{!}_i$ for $i\in\{1,\ldots, n-1\}$.

The inclusion
$$
\iota: Y_s \to X^{!}
$$
induces a map
$$
\iota^{*}: \Pic(X^{!})\otimes_{\mathbb{Z}}\mathbb{R} \to \Pic(Y_s)\otimes_{\mathbb{Z}}\mathbb{R}
$$
Define $\mathfrak{D}_+(Y_s)=\iota^{*}(\mathfrak{D}_+(X^{!}))$. Slopes in $\mathfrak{D}_{+}(Y_s)$ will be the right choice for the stable envelopes of $Y_s$.

Since the torus $\bA^{!}$ acts on $Y_s$, the chamber $-\mathfrak{C}$ automatically gives a chamber for the stable envelopes of $Y_s$. 

The $\nu_s$-invariant part of the polarization for $X^{!}$ gives a polarization $\left(T^{1/2}X^{!}\right)^{\nu_s}$ of $Y_s$. 

With respect to these choices, we can now talk about the $K$-theoretic stable envelopes of $Y_{s}$.

\begin{Theorem}\label{mainthm2}
Let $s\in \Wall(X)$ be big enough and let $\epsilon\in \mathfrak{D}_+(X)$ be a small ample slope of $X$ such that $s'=s+\epsilon$ is a generic slope. Then
$$
\Ind^{!s} =\kappa\left( \mathsf{D} \cdot \mathsf{H} \cdot \widetilde{A}^{\mathfrak{D}_+(Y_{s+1}),Y_{s+1}}\cdot \mathsf{H}^{-1}  \cdot \widetilde{A}^{s',X} \cdot \mathsf{D}^{-1} \cdot \sympow\left( \hbar \otimes N^+ \right) \right)
$$
where $\mathsf{H}$ is the diagonal matrix given by
$$
\mathsf{H}= \text{diag}\left((-1)^{\gamma_{I}(s)} \hbar^{m_I(s)/2} \prod_{i=1}^{n-1} \mathscr{L}_i|_{I}\right)_{I \in X^{\bT}}
$$
where $\gamma_I(s)$ and $m_I(s)$ are integers.
\end{Theorem}
\begin{proof}
From Theorem 3 and Theorem 5 in \cite{KS2}, the twisted limit of the restriction matrix of the elliptic stable envelope is
$$
\lim_{q\to 0} \widetilde{T}^{X}|_{z=z q^{s+1}} = \widetilde{Z}^{'} \cdot \widetilde{A}^{s'+1,X}
$$
Here, the matrix $\widetilde{Z}^{'}$ is given by
$$
\widetilde{Z}^{'}= \mathsf{H} \cdot \widetilde{A}^{\mathfrak{D}_{+}(Y_{s+1}),Y_{s+1}}\cdot \mathsf{H}^{-1}
$$
where $\mathsf{H}$ is the diagonal matrix 
$$
\mathsf{H}=\text{diag} \left( (-1)^{\gamma_{I}(s)} \hbar^{m_{I}(s)/2} \left(\prod_{i=1}^{n-1} \mathscr{L}_i|_{I}\right) \right)_{I \in X^{\bT}}
$$ 
for integers $\gamma_{I}(s)$ and $m_{I}(s)$. For the precise description of these integers, see Theorem 2 in \cite{KS1} and Theorem 4 in \cite{KS2}.

The limits of the rest of the terms in (\ref{msflags}) are the same as in Theorem \ref{mainthm}, which finishes the proof.
\end{proof}

Contrary to the case of generic $s$, the right hand side of Theorem \ref{mainthm2} does not admit a nice interpretation as an equivariant Euler characteristic of a simple twist of the $K$-theoretic stable envelope of $X$. However, we still have:

\begin{Corollary}
For any big enough $s \in  \Lie_{\mathbb{R}}(\bK)$, the index vertex $\Ind^{! s}_{I^{!}}$ is a rational function of $\aa^{!}$, $\zz^{!}$, and $\sqrt{\hbar^{!}}$.
\end{Corollary}

\section{Example: \texorpdfstring{$T^*\mathbb{P}^1$}{P}}
We work out our main result for generic slopes explicitly here in the simplest example: $X=T^*\mathbb{P}^1$. This corresponds to the quiver shown in Figure \ref{fig2}.

\begin{figure}[ht]
\centering
\begin{tikzpicture}[roundnode/.style={circle,fill,inner sep=2.5pt},squarednode/.style={rectangle,fill,inner sep=3pt}] 
\node[squarednode,label=below:{$2$}](F1) at (0,-1.5){};
\node[roundnode,label=above:{$1$}](V1) at (0,0){};
\draw[thick, ->] (V1) -- (F1);
\end{tikzpicture}
\caption{The quiver data for the variety $T^*\mathbb{P}^1$.} \label{fig2}
\end{figure}
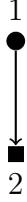
Let $V=\mathbb{C}$ and $W=\mathbb{C}^2$. Points in $X$ are represented by maps
$$
(I,J) \in Hom(W,V)\oplus Hom(V,W)
$$
such that $I$ is injective and $I \circ J=0$, modulo the action of $GL(1)$ on $V$. The torus $(\mathbb{C}^{\times})^{2}$ acts on $\mathbb{C}^2$ by
$$
(u_1,u_2) \cdot (x_1,x_2)=(u_1^{-1}x_1,u_2^{-1}x_2)
$$
which induces an action of $\bA=\mathbb{C}^{\times}$, where the coordinate on this torus is $a=u_1/u_2$.

The torus $\mathbb{C}^{\times}_{\hbar}$ acts by
$$
\hbar \cdot (I,J) = (\hbar^{-1}I,J)
$$
We denote $\bT=\bA\times \mathbb{C}^{\times}_{\hbar}$. There are two $\bT$-fixed points $p_1$ and $p_2$, which are labeled by the identity permutation and the simple transposition, respectively. The vector space $V$ descends to a line bundle $\mathcal{V}$ on $X$ and the $\bT$-character of this tautological bundle at the fixed point $p_i$ is $u_i$.

We denote the dual copy of $T^*\mathbb{P}^1$ as $X^{!}$ and denote the torus as $\bT^{!}=\bA^{!}\times \mathbb{C}^{\times}_{\hbar^{!}}$. The exchange of equivariant and K\"ahler parameters is given by (\ref{kap}) and the bijection on fixed points is the identity map.

\subsection{Torus fixed quasimaps}


A quasimap $f\in \qm^{d}_{p_i}$ to $X^{!}$ consists of a degree $d$ line bundle $\mathscr{V}$ and a trivial rank 2 vector bundle $\mathscr{W}$ over $\mathbb{P}^1$, along with a section
$$
f\in  H^{0}(\mathbb{P}^1,\mathscr{M} \oplus (\hbar^{!})^{-1} (\mathscr{M})^{\vee})
$$
where 
$$
\mathscr{M}= Hom(\mathscr{V},\mathscr{W})
$$
Assume the quasimap is $\bT^{!}_{q}$-fixed. Equivariantly with respect to the torus $\bT^{!}$, we have
$$
\mathscr{W}={u_1^{!}} \mathcal{O}_{\mathbb{P}^1} \oplus {u_2^{!}} \mathcal{O}_{\mathbb{P}^1}
$$
where $w \mathcal{O}_{\mathbb{P}^1}$ denotes a twist of the structure sheaf of $\mathbb{P}^1$ by a trivial line bundle with weight $w$. Since the section $f$ takes the value $p_i$ at $\infty$, in order $f$ to be $\bT^{!}$-fixed, $\mathscr{V}$ must have $\bT^{!}$-weight ${u_i^{!}}$. Furthermore, the component of the section in 
$$
Hom(\mathscr{V},u_i^{!}\mathcal{O}_{\mathbb{P}^{!}})\cong \mathcal{O}(-d)
$$
must be nonzero at $\infty$. So $d \leq 0$ and there is only one such section, which is given in homogeneous coordinates by $x_0^{-d}$. To be invariant under $\mathbb{C}^{\times}_q$, the bundle $\mathscr{V}$ must be further twisted by the trivial line bundle with $\mathbb{C}^{\times}_q$-weight of $q^{-d}$. This completely determines the quasimap.

To summarize, there is a unique $\bT^{!}_{q}$-fixed quasimap in $\qm^{d}_{p_i}$ for each $d \leq 0$, and the data is as follows:
\begin{itemize}
    \item $\mathscr{W}=u_1^{!} \mathcal{O}_{\mathbb{P}^{!}} \oplus u_2^{!} \mathcal{O}_{\mathbb{P}^{!}} $
    \item $\mathscr{V}=u_i^{!} q^{-d} \mathcal{O}(-d)$
    \item $f=x_0^{-d}$
\end{itemize}

\subsection{Virtual tangent space}
Fix $d \leq 0$ and let $\mathscr{W}$, $\mathscr{V}$, and $f_d$ be the unique $\bT^{!}_q$-fixed quasimap of degree $d$ as above. We have the induced virtual bundle from (\ref{Tbundle}):
\begin{equation}\label{qmpolp1}
\mathcal{T}^{1/2}=\mathscr{V}^{\vee} \otimes \mathscr{W} - \mathscr{V}^{\vee}\otimes \mathscr{V} = q^{-d} \mathcal{O}(-d) + u^{!} q^{-d} \mathcal{O}(-d)-\mathcal{O}_{\mathbb{P}^1}
\end{equation}
where 
$$
u^{!}= \frac{u_j}{u_i}, \quad j \in \{1,2\}\setminus \{i\}
$$

The reduced virtual tangent space at $f_d$ is
$$
\mathscr{T}^{\mathrm{red}}_{\mathrm{vir},f_d}=H^*(\mathcal{T}^{1/2}\oplus \hbar^{-1}(\mathcal{T}^{1/2})^{\vee})- T_{p_i} X
$$
The action of $\mathbb{C}^{\times}_q$ in (\ref{Cq}) is such that the global sections of $\mathcal{O}(m)$ for $m\geq 0$ have character $1+q^{-1}+\ldots+q^{-m}$. Thus, any given term $x q^{-d} \mathcal{O}(-d)$ in $\mathscr{T}^{1/2}$ of (\ref{qmpolp1}) leads to contributions of the form
\begin{multline*}
H^*\left(x q^{-d} \mathcal{O}(-d)+ (\hbar^{!})^{-1} x^{-1} q^{d} \mathcal{O}(d)\right)-x-(\hbar^{!})^{-1}x^{-1} \\=
xq\left(1+q+\ldots+q^{-d-1}\right)-\frac{1}{\hbar^{!} x}\left(1+q^{-1}+\ldots +q^{d+1}\right) 
\end{multline*}
So the $\bT^{!}_q$-character of the reduced virtual tangent space at $f$ is
\begin{align} \nonumber\label{tvir}
\mathscr{T}^{\mathrm{red}}_{\mathrm{vir},f_d}= q\left(1+q+\ldots + q^{-d-1}\right)&-\frac{1}{\hbar^{!}}\left(1+q^{-1}+\ldots+q^{d+1}\right)  \\ + q u^{!}\left(1+q+\ldots+q^{-d-1}\right)&-\frac{1}{\hbar^{!} u^{!}}\left(1+q^{-1}+\ldots+q^{d+1}\right)
\end{align}

\subsection{Vertex functions}

The twist in \cite{pcmilect} section 6.1.8 that transforms the virtual structure sheaf to the symmetrized virtual structure sheaf means that the vertex function from Theorem \ref{qmver} is
$$
\sum_{d \leq 0} (z^{!})^{-d} q^{\deg \mathcal{T}^{1/2}/2} \hat{a}\left(\mathscr{T}^{\mathrm{red}}_{\mathrm{vir},f_d}\right)
$$
Writing this out explicitly, we find that for the fixed point $p_i$ with $u^{!}$ as above, we have
\begin{align*}
\sum_{d \leq 0} (z^{!})^{-d} q^{\deg \mathcal{T}^{1/2}/2} \hat{a}\left(\mathscr{T}^{\mathrm{red}}_{\mathrm{vir},f_d}\right)&=\sum_{d = 0}^{\infty} (z^{!})^{d} q^{d} \frac{(\hbar^{!})_d }{(q)_d}\left(-\frac{q}{\hbar^{!}}\right)^{d/2} \frac{(\hbar^{!} u^{!})_d}{ (q u^{!})_d} \left(-\frac{q}{\hbar^{!}}\right)^{d/2} \\
&=\sum_{d \geq 0} \left(\frac{q^2}{\hbar^{!}} z^{!} \right)^d \frac{(\hbar^{!})_d }{(q)_d} \frac{(\hbar^{!} u^{!})_d}{ (q u^{!})_d}
\end{align*}

The vertex functions from Definition \ref{ver} are
\begin{align*}
    V_{p_1}^{!}(a^{!},z^{!})=\sum_{d=0}^{\infty} \frac{(\hbar^{!})_d (\hbar^{!}u_2^{!}/u_1^{!})_d}{(q)_d (q u_2^{!}/u_1^{!})_d} (z^{!})^{d} \\
   V_{p_2}^{!}(a^{!},z^{!})=\sum_{d=0}^{\infty} \frac{(\hbar^{!})_d (\hbar^{!}u_1^{!}/u_2^{!})_d}{(q)_d (qu_1^{!}/u_2^{!})_d} (z^{!})^{d} 
\end{align*}
which clearly agree with the previously ones after renormalizing the K\"ahler parameter.

\subsection{Index}
With respect to the choice of chamber (\ref{chamb}), the attracting weight is $a^{!}:=u^{!}_1/u^{!}_2$. Write $f_{d}^i$ for the unique $\bT^{!}_q$-fixed quasimap of degree $d$ on $\qm^{d}_{p_i}$. The polarization of the reduced virtual tangent space (\ref{tvir}) for $\qm^{d}_{p_1}$ at $f^{1}_d$ is
$$
\mathscr{T}^{1/2}_{f^{1}_d}=q\left(1+q+\ldots+q^{-d-1}\right) +\frac{q}{a^{!}}\left(1+q+\ldots+q^{-d-1}\right)
$$
Thus the $s$-index from Definition \ref{sindex} is
$$
\mathcal{I}_s\left(f\right) = \begin{cases}
 -d & -d <s  \\
 \lfloor s \rfloor & 1<s<-d \\
0 & s<1
\end{cases}
$$
where $\lfloor \cdot \rfloor$ denotes the floor function. Hence
\begin{align*}
\Ind^{! s}_{p_1}&=\begin{dcases} 
\frac{1-\left(z^{!} \hbar^{!}/q\right)^{\lfloor s\rfloor +1 }}{1-z^{!} \hbar^{!}/q} + \frac{(\hbar^{!}/q)^{\lfloor s \rfloor } z^{! \lfloor s \rfloor +1}}{1-z^{!}} & s>1 \\
\frac{1}{1-z^{!}} & s <1
\end{dcases}
\end{align*}
Similarly, the polarization of the virtual tangent space (\ref{tvir}) for $\qm^{d}_{p_2}$ at $f^{2}_d$ is
$$
\mathscr{T}^{1/2}\big|_{p_2}=q\left(1+q+\ldots+q^{-d-1}\right) +a^{!}q\left(1+q+\ldots+q^{-d-1}\right)
$$
So
$$
\mathcal{I}_s\left(f\right) = \begin{dcases}
0 & s>-1 \\
 \lfloor |s| \rfloor & d < s<-1 \\
-d=|d|  & s<d
\end{dcases}
$$
and
\begin{align*}
\Ind^{! s}_{p_2}&=\begin{dcases} 
\frac{1-\left(z^{!} \hbar^{!}/q\right)^{\lfloor | s | \rfloor +1}}{1-z^{!} \hbar^{!}/q} + \frac{(\hbar^{!}/q)^{\lfloor | s | \rfloor } z^{! \lfloor | s | \rfloor +1}}{1-z^{!}} & s<-1 \\
\frac{1}{1-z^{!}} & s >-1
\end{dcases}
\end{align*}

\subsection{Stable envelopes}
Now we switch from $X^{!}$ to $X$. We need to compute the $K$-theoretic stable envelopes for $X$. With respect to the chamber $\mathfrak{C}$ from (\ref{chamb}), $a:=u_1/u_2$ is an attracting weight. The choice of polarization is assumed to be 
$$
T^{1/2}X=Hom(V,W)-Hom(V,V)
$$
Under (\ref{lierk}), big enough slopes are given by $s>0$.

Analogous to the argument given in section 7.1 of \cite{OS}, one can use the defining properties of stable envelopes to calculate the matrix of the $K$-theoretic stable envelope in the basis of fixed points ordered as $[p_1,p_2]$ to be:
\begingroup
\renewcommand*{\arraystretch}{2}
$$
A^{s,X}=\left(\stab^{s,X,K}_{-\mathfrak{C}}(I)|_{J}\right)_{I,J\in X^{\bT}}=
\begin{pmatrix}
\left(\frac{a}{\hbar}-1\right)\sqrt{\hbar} & \left( \hbar^{-1}-1\right) \sqrt{\hbar} a^{\lfloor s \rfloor}\\
0 & (1-a^{-1})
\end{pmatrix}
$$ 
\endgroup
Alternatively, one could use the formulas \cite{dinkinselliptic} and their implementation in Maple described there. Normalizing by dividing each column by the diagonal entry gives
\begingroup
\renewcommand*{\arraystretch}{2}
$$
\widetilde{A}^{s,X}=
\begin{pmatrix}
1 & -\dfrac{\left( 1-\hbar\right)a^{\lfloor s \rfloor+1}}{(1-a)\sqrt{\hbar}}\\
0 & 1
\end{pmatrix}
$$ 
\endgroup
The polarization at the fixed points is
\begin{align*}
    T^{1/2}X|_{p_1}&=a^{-1} \\
    T^{1/2}X|_{p_2}&=a
\end{align*}
So conjugating by $\mathsf{D}$ in (\ref{Dmat}) and multiplying on the right by $\sympow\left(\hbar\otimes N^{+}\right)$ gives
\begingroup
\renewcommand*{\arraystretch}{2}
$$
\mathsf{D}\cdot \widetilde{A}^{s,X} \cdot \mathsf{D}^{-1} \cdot \sympow(\hbar \otimes N^+)=
\begin{pmatrix}
\dfrac{1}{1-a}  -\dfrac{\left( 1-\hbar\right)a^{\lfloor s \rfloor}}{(1-a)(1-\hbar a)}\\
 \dfrac{1}{1-\hbar a}
\end{pmatrix}
$$ 
\endgroup
Applying $\kappa$ and summing the two components of this vector, we see that Theorem $\ref{mainthm}$ for $s>1$ reads
\begin{align*}
   \frac{1-\left(z^{!} \hbar^{!}/q\right)^{\lfloor s\rfloor +1 }}{1-z^{!} \hbar^{!}/q} + \frac{(\hbar^{!}/q)^{\lfloor s \rfloor } z^{! \lfloor s \rfloor +1}}{1-z^{!}} &=\frac{1}{1-z^{!} \hbar^{!}/q}  -\frac{\left( 1-q/\hbar^{!}\right)\left(z^{!} \hbar^{!}/q\right)^{\lfloor s+1 \rfloor}}{\left(1-z^{!} \hbar^{!}/q\right)\left(1-z^{!}\right)} \\
 \frac{1}{1-z^{!}}&= \frac{1}{1-q/\hbar^{!} ( z^{!} \hbar^{!}/q)} 
\end{align*}
which is easily checked to be true. For $0<s<1$, Theorem \ref{mainthm} reads
\begin{align*}
    \frac{1}{1-z^{!}}&=\frac{1}{1-z^{!}\hbar^{!}/q}-\frac{\left(1-q/\hbar^{!}\right)\left(z^{!}\hbar^{!}/q\right)}{\left(1-z^{!}\hbar^{!}/q\right)\left(1-z^{!}\right)} \\
    \frac{1}{1-z^{!}}&= \frac{1}{1-q/\hbar^{!} ( z^{!} \hbar^{!}/q)} 
\end{align*}
which is also true.

\printbibliography

\newpage

\noindent
Hunter Dinkins\\
Department of Mathematics,\\
University of North Carolina at Chapel Hill,\\
Chapel Hill, NC 27599-3250, USA\\
hdinkins@live.unc.edu 

\vspace{3 mm}

\noindent
Andrey Smirnov\\
Department of Mathematics,\\
University of North Carolina at Chapel Hill,\\
Chapel Hill, NC 27599-3250, USA\\
Steklov Mathematical Institute \\
of Russian Academy of Sciences, \\
Gubkina str. 8, Moscow, 119991, Russia \\
asmirnov@email.unc.edu

\end{document}